\documentclass[12pt]{amsart}

\usepackage{graphicx}
\usepackage{amssymb,amsmath}

\newtheorem{theorem}{Theorem}[section]

\newtheorem{claim}[theorem]{Claim}
\newtheorem{fact}[theorem]{Fact}
\newtheorem{corollary}[theorem]{Corollary}
\newtheorem{proposition}[theorem]{Proposition}

\theoremstyle{definition}
\newtheorem{definition}[theorem]{Definition}
\newtheorem{question}[theorem]{Question}

\newtheorem{example}[theorem]{Example}

\theoremstyle{remark}
\newtheorem{remark}[theorem]{Remark}

\numberwithin{equation}{section}

\newcommand{\cE}{\mathcal{E}}

\newcommand{\cA}{\mathcal{A}}

\newcommand{\cB}{\mathcal{B}}

\newcommand{\cG}{\mathcal{G}}

\newcommand{\bP}{\mathbb{P}}

\newcommand{\bE}{\mathbb{E}}
\newcommand{\N}{\mathbb{N}}

\newcommand{\R}{\mathbb{R}}

\newcommand{\ga}{\gamma }
\newcommand{\Ga}{\Gamma }

\newcommand{\Om}{\Omega }

\newcommand{\Energy}{\mathop{\mathrm{Energy}}}
\newcommand{\Ker}{\mathop{\mathrm{Ker}}}
\newcommand{\diam}{\mathop{\mathrm{diam}}}
\newcommand{\Diag}{\mathop{\mathrm{Diag}}}

\newcommand{\epdist}{\mathop{\mathrm{Epidemic}}}
\newcommand{\length}{\mathop{\mathrm{length}}}

\newcommand{\Div}{\mathop{\mathrm{div}}}

\newcommand{\Dim}{\mathop{\mathrm{Dim}}}

\newcommand{\Span}{\mathop{\mathrm{Span}}}

\newcommand{\Mod}{\mathop{\mathrm{Mod}}}
\newcommand{\capa}{\mathop{\mathrm{Cap}}}
\newcommand{\rlength}{\mathop{\rho\mathrm{-length}}}

\newcommand{\defeq}{\mathrel{\mathop:}=}

\newcommand{\cReff}{\mathop{\mathcal{R}_{eff}}}
\newcommand{\cRpeff}{\mathop{\mathcal{R}^{\prime}_{eff}}}
\newcommand{\cCeff}{\mathop{\mathcal{C}_{eff}}}

\title{Effective resistance on graphs and the Epidemic quasimetric}

  \author[Ericson, Poggi-Corradini, and Zhang]
         {Josh Ericson, Pietro Poggi-Corradini, and Hainan Zhang\\
         Department of Mathematics, Cardwell Hall, Kansas State University,
Manhattan, KS 66506, USA }

\thanks{The first and last authors completed this project while undergraduate students at Kansas State University. The second author wants to thank the Center for Engagement and Community
Development at Kansas State University, the Department of Electrical
and Computer Engineering at Kansas State University, and the
I-center in the Department of Mathematics at Kansas State University,
for their generous support of this research. This project was supported in part by NSF grant n. 1201427}

\date{\today}     
\begin{document}
\begin{abstract}
We introduce the epidemic quasimetric on graphs and study its behavior
with respect to clustering techniques. In
particular we compare its behavior to known objects such as the graph
distance, effective resistance, and modulus of curve families.
\end{abstract}

\maketitle
\baselineskip=18pt

\section{Introduction} 

This study was initiated by the need to analyze real world data collected in the rural town of Chanute, Kansas\footnote{This was a joint project of the first author with Prof. C.~Scoglio in the Dept. of Electrical and Computing Engineering and Prof. W.~Schumm in the Department of Family Studies, at Kansas State University}. The goal was to study and simulate potential epidemic outbreaks. From the survey, a contact network was constructed representing the sampled population and their potential relationships. Mathematically, this is just a graph where the vertices represent people and the edges represent possible interactions. In this paper, we introduce a new geometric quantity, the epidemic quasimetric, which we study and relate to more classical quantities, such as effective resistance.

One of the simplest geometric object that is used to study finite graphs is the ``graph metric''. Namely, the graph metric measures the distance between two nodes $a$ and $b$ by computing the minimal number of edges that must be traversed (`hops') to go from $a$ to $b$. 

Epidemics, on the other hand, can be modeled to begin at one node, then spread to all the neighbors, and then to all the neighbors' neighbors, etc...The possible damage of the epidemic spreads as a circular wave. We use this dynamic to assign to every pair of nodes of a finite graph a number, which we call the``epidemic quasimetric''. To compute the epidemic quasimetric between $a$ and $b$ we expand the range of an epidemic started at $a$ until $b$ is affected and compute the number of edges that became involved in the process. Then we do the same interchanging $a$ and $b$, and we add the two numbers thus obtained. This is fairly easy to compute numerically and we describe the routine we implemented in Matlab in Section \ref{sec:epidd}.

Part of the inspiration for considering the epidemic quasimetric came from reading \cite{semmes1993}\footnote{Diego Maldonado pointed out that a similar concept had been introduced earlier in \cite{macias-segovia1979}} where a similar quantity is introduced in order to study bi-Lipschitz embeddings of metric spaces. The hope is that the epidemic quasimetric contains geometric information that allows to view the graph under a new light. To partially confirm this intuition, we experimented with the epidemic quasimetric and showed how it can be used to obtain a pretty accurate cut of the famous Karate Club graph into ``natural'' communities, see Section \ref{sec:numerical}. In this direction, aside for clustering techniques, the epidemic quasimetric could be also be useful in sparsification techniques.

Our second goal is to compare the notion of epidemic quasimetric to the more classical notion of effective conductance, when the graph is viewed as an electrical network. Effective conductance has also been used in the literature to study graphs from the point of view of  community detection and sparsification. Therefore, such a comparison gives us hope that the epidemic quasimetric can also be used effectively to study graphs while being relatively simple to compute numerically. 

The paper begins with some preliminaries and notations about graphs; then  in Section \ref{sec:epidd} we define the epidemic quasimetric and state our goal to relate it to effective resistance. Thereafter we review the theory of random walks on finite graphs in Section \ref{sec:markov}, its connection to electrical networks, and the notion of effective conductance in Section \ref{sec:electrical}. Some references for these sections are \cite{doyle-snell1984}, \cite{peres2009}, and \cite{grimmett2010}.

Then, in Section \ref{sec:capmod}, we introduce two more concepts drawn from modern geometric function theory. Namely the notions of capacity and modulus of families of curves, see \cite{ahlfors1973}. In Theorem \ref{thm:modcap} we show that all of these concepts coincide with the notion of effective conductance (the method of Lagrange Multipliers turned out to be useful in this context). Moreover, in Proposition \ref{prop:upperbound} we exploit the definition of modulus to obtain a comparison between modulus and epidemic quasimetric. Thus we get an estimate for the epidemic quasimetric in terms of  effective conductance. 

Finally, in Section \ref{sec:numerical} we describe our numerical computations and experiments.

We begin with some preliminaries on elementary graph theory.

\section{Graphs}\label{sec:graphs}

\subsection{Notation and generalities}\label{ssec:graphs}
We will restrict our study to simple, finite, connected graphs.
Let $G= (V,E)$ be a {\em graph} with vertex-set $V$ and edge-set $E$.
We say that $G$ is {\em simple} if there is at most one undirected edge between
any two distinct vertices, and it is {\em finite} if the vertex set has
cardinality $|V|=N\in \N$. In this case, the edge-set $E$ can be
thought of as a subset of ${V\choose 2}$, the set of all unordered
pairs from $V$. Therefore the cardinality of $E$ is $M=|E|\leq
{N\choose 2}= (N/2) (N-1)$.

We say that two vertices $x,y$ are {\em neighbors} and write $x\sim y$
if $\{x,y \}\in E$. The graph is {\em connected} if for any two vertices
$a,b\in V$ there is a chain of vertices $x_{0}=a,x_{1},\dots,x_{n}=b$,
so that $x_{j}\sim x_{j+1}$ for $j=0,\dots, n-1$. It is known that
connected graphs must satisfy $|E|\geq N-1$ (induction).

Given a subset of vertices $V^{\prime}\subset V$, we let $E
(V^{\prime})\subset E$ be all the edges of $G$ that connect pairs of
vertices in $V^{\prime}$. With this notation $G (V^{\prime})=
(V^{\prime}, E (V^{\prime}))$ is a simple graph which we call the 
{\em subgraph induced by $V^{\prime}$}. More generally,    
a {\em subgraph} of $G$ is a graph $G^{\prime}= (V^{\prime},
E^{\prime})$ such that $V^{\prime}\subset V$ and $E^{\prime }\subset E
(V^{\prime})$.

The number of edges that are incident at a vertex $x$ is called the
{\em degree} of $x$ and we write $d (x)$. Since every edge is incident
at two distinct vertices, it contributes to two degrees. Therefore
\[
\sum_{x\in V}d (x)=2|E|.
\]
This identity is sometimes referred to as the ``Handshake Lemma''.
It says that instead of counting edges, one can add degrees,
i.e., switch to $d (x)$ which is a function defined on $V$.

For instance, {\em the volume} of a subgraph $H= (V (H), E (H))$ of
$G$ can be defined as $|E (H)|$, i.e.,
the number of edges of $H$; or as half the sum of the $H$-degrees over
the vertices of $H$: 
\[
\frac{1}{2}\sum_{x\in V (H)}d_{H} (x). 
\]

We say that $\ga$ is a {\em curve} in $G$ if $\ga$ is a connected
subgraph of $G$. This is not a very common way of defining curves in graph theory, but it makes sense from the point of view of the function-theory inspired concepts that we will introduce later when we talk about modulus of curve families. 

The graph $G$ is {\em weighted} if there is a weight function $W:E\rightarrow[0,+\infty)$ defined on the edges. The unweighted graph is recovered by setting $W_0(e)= 1$, $\forall e\in E$.

Given a curve $\ga$ in $G$,  it is natural to define
its {\em graph-length} to be the total number of edges in $\ga$, i.e.,
\begin{equation}\label{eq:graphlengthw}
\length_G(\ga)\defeq \sum_{e\in E(\ga)}W_0(e).
\end{equation}

\subsection{The graph distance}\label{ssec:grdist}

A function of two variables $d(x,y)\geq 0$ on a space $X$ is called a
{\em metric} or a {\em distance}, if it is symmetric, $d(x,y)=d(y,x)$;
non-degenerate, $d(x,y)=0$ iff $x=y$; and satisfies the {\em triangle
inequality}, $d(x,y)\leq d(x,z)+d(z,y)$, whenever $x,y,z\in X$.  

On a connected graph $G$  the {\em graph distance} $d_G(x,y)$ is
defined as the shortest graph-length of a curve connecting $x$ to $y$. 
\begin{equation}\label{eq:graphdist}
d_G(x,y)\defeq\min_{\ga:x\leadsto y} \length_G(\ga),
\end{equation}
where $\ga:x\leadsto y$ means that $\ga$ is a curve connecting $x$ to $y$. We leave the verification that $d_G$ is a metric to the reader.

The {\em diameter} of the graph $G$ is 
\[
\diam(G)=\max_{x,y\in V} d_G(x,y).
\]

The {\em metric ball} centered at a vertex $x$ and of radius $r$ is
\[
\cB (x,n)\defeq \{y\in V: d_{G} (x,y)\leq r \}.
\]

\section{The epidemic quasimetric}\label{sec:epidd}

Given two nodes $x$ and $y$, we define the epidemic quasimetric between
them to be the size of the part of the graph that would be affected
(the potential damage) if an epidemic started at $x$ and reached $y$,
or vice versa. In formulas, we consider all the vertices in 
\[
\cB (x,d_{G} (x,y))
\]
and form the induced subgraph of $G$ which we call $\Omega (x,d_{G} (x,y))$. 
We then compute the volume $|\Omega (x,d_{G} (x,y))|$ as in Section \ref{ssec:graphs}, by
counting the
number of edges. So we define the {\em epidemic quasimetric} between $x$
and $y$ to be
\begin{equation}\label{eq:epidemic}
\epdist (x,y) \defeq |\Omega (x,d_{G} (x,y))|+|\Omega (y,d_{G} (x,y))|.
\end{equation}

The epidemic quasimetric is not a distance in the mathematical sense. For
instance, the triangle inequality can fail
as badly as possible, as the following example shows.

\begin{example}\label{ex:trineq}
In Figure \ref{fig:trineq},
\[
\epdist (x_{1},x_{2})=3 \qquad \mbox{and}\qquad \epdist (x_{2},x_{3})= 4
\]
while
\[
\epdist (x_{1},x_{3})=N+5.
\]
So the triangle inequality in this case can be made to fail as badly
as needed by letting $N$ increase.
\begin{figure}[h]
\includegraphics{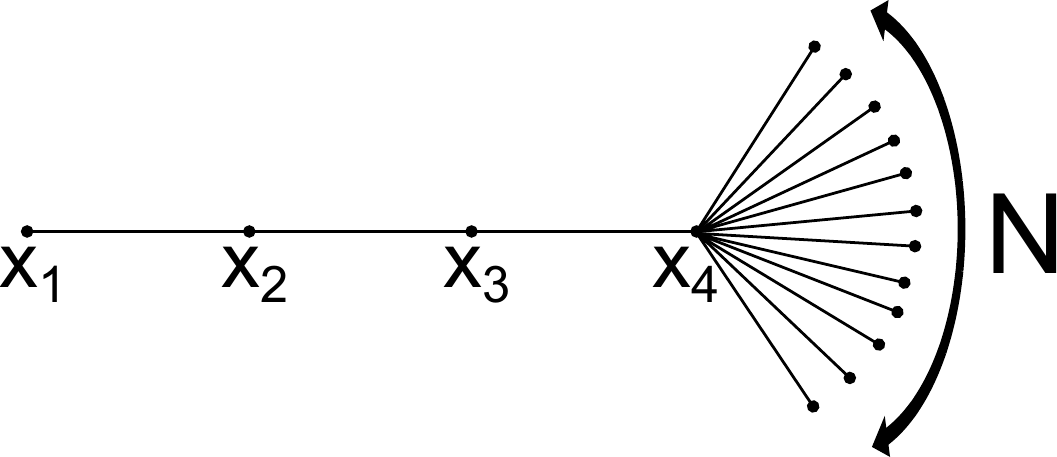}
\caption{Failure of the triangle inequality}\label{fig:trineq}
\end{figure}
\end{example}

As described in Section \ref{sec:numerical}, the epidemic quasimetric seems to carry useful geometric information about a graph and in our experiments it appears to behave well with respect to simple clustering algorithms. For the rest of the paper, our intent is to present a comprehensive survey of effective resistance and to answer the following question.

\begin{question}
Is there a connection between the epidemic quasimetric and effective resistance?
\end{question}

We begin in the next two sections by surveying Markov
chains and electrical networks. Most of this material can be found in
\cite{peres2009}, \cite{grimmett2010}, \cite{doyle-snell1984}.

\section{Markov Chains} \label{sec:markov}

A Markov chain is comprised of a finite set $S$ (the state space)
and a probability distribution on the state space which can be
represented in terms of a {\em transition matrix} $\{ P(i,j) \} = P$ whose
entries correspond to the probability of being at state $j$ at time
$n+1$ given that you were at state $i$ at time $n$. $P$ being a
transition matrix means that we must have $\sum_{j} P(i,j) =1$ (Row
sums add up to $1$.)  

This determines a process, i.e., a sequence of random variables 
$X_n=\Omega \rightarrow S$, with the property that $\mathbb{P}
(X_{n+1}=y | X_n = x) = p(x,y)$. Note that $P(x,y)$ is independent of
$n$.

\example{A random walk on a graph is an example of a Markov chain.
Consider a finite, simple graph $G=(V,E) $ with vertices $V=\{1, 2,...,
N\} $ and edges $E$. For $x \in V$, $d_x$
indicates the degree or local index of $x$. 
We've seen that, by the ``hand-shake'' lemma,  $\sum_{x \in V} d_x =2 |E|$.

\noindent Define

\begin{equation}\label{eq:grphtrans}
   P (x,y) \defeq  \left\{
     \begin{array}{lr}
       \frac{1}{d_{x}} &  x \sim y\\
       0 &  \mbox{ else. }
     \end{array}
   \right.
\end{equation}

\vspace*{.2in}

\noindent Note that $\sum_{y \in V}P(x,y)=\sum_{y \sim x} 1/d_x = (1/d_x)
\sum_{y \sim x}1 = 1$.} 
 
\subsection{Matrices act on column vectors} Given $N$ states, $\{P(x,y)\}=P$ is an $N
\times N$ matrix. Here $N \times 1$ vectors correspond to functions
$f: S \rightarrow \mathbb{R}$ on the state space. Then $P$ acts on
functions as follows: \[g(x)=(Pf)(x)=\sum_{y \in S} P(x,y)f(y).\]
Probabilistically, \[g(x)= \sum_{y \in S} f(y)
\mathbb{P}(X_{1}=y|X_0=x)=\mathbb{E}_{X_0=x}(f(X_1)),\] which is the
average or expected value of $f$ evaluated on the process at time 1.

\noindent On a graph: \[g(x)=\frac{1}{d_x} \sum_{y\sim x}f(y).\] So $g$ is
obtained from $f$ by defining $g (x)$ to be the average of the values
of $f$ over all the 
neighbors of node $x$.

\definition{Whenever $x\in V$ and $(Pf)(x)=f(x)$ we say that $f$ is a
{\em harmonic} function on the graph at $x$.} 

\example[\bf Gambler's Ruin]{Consider six nodes $\{0,1,2,3,4,5\}$
representing the dollar amounts held by a gambler. After each bet the
gambler either wins or loses a dollar with equal probability. The
gambler will walk away whenever his fortune is $0$ (ruin), or $5$
(predetermined goal). We represent this with a transition matrix $P$,
which is  a $6\times 6$ matrix in this case. 

We would like to know the probability of reaching $5$ before $0$, assuming
that we start our random walk at node $3$. This is an example of a
 ``hitting probability'', what in complex analysis would be called a
 ``Harmonic Measure'' problem. 

\noindent Let \[h(x)= \bP(\mbox{$X_n$ hits $5$ before $0$} |
X_0=x)=\bP_x(\mbox{$X_n$ hits $5$ before $0$}),\]  
where $\bP_x$ is probability conditioned on $\{X_0=x\}$. 

\noindent We want to compute $h(3)$. We call  $B=\{0, 5\}$ {\em boundary
points} while $I=\{1, 2, 3, 4\}$ are called {\em interior points}. 
}
For $x\in
I$, we can condition on the first step; this is known as ``first-step
analysis''. For brevity we write $A\defeq \{\mbox{$X_n$ hits $5$ before
$0$}\}$:
\begin{eqnarray*}
h(x) & = & \bP_x(A) \\
       & = & \bP_x(A | X_1=x-1)\bP_x(X_1=x-1)+ \bP_x(A| X_1=x+1)\bP_x(X_1=x+1)\\
      & = & \frac{1}{2} h(x-1)+\frac{1}{2}h(x+1)\end{eqnarray*} since
$P(x,y)=1/2$ if $y=x-1$ or $x+1$, and $0$ otherwise. The function $h$
is thus harmonic for $P$ at each interior point. Also $h$
has boundary values $h(0)=0$ and $h(5)=1$.

\begin{fact}[Maximum Principle]\label{fact:maxpr}
A harmonic function achieves its maximum
value $M$ and minimum value $m$ on the boundary.
\end{fact}

Idea: suppose $f$ attains its maximum at an interior point. By
harmonicity it must attain it at each neighboring vertex too and this
will propagate out as an oil spill all the way to the boundary.

\begin{claim}\label{cl:uniq}The problem of finding a harmonic function $h$  on $I$ with boundary values $h(0)=0$ and $h(5)=1$  has a unique solution.
\end{claim} 

\begin{proof}Suppose that $h$ and
$g$ are both harmonic at each interior point and that $h(0)=g(0)=0,
h(5)=g(5)=1.$ Let $f= h-g$. Then, by linearity, we have that
$$f(x)=\frac{1}2f(x-1)+\frac{1}2f(x+1)$$ with $f(0)=0$ and $f(5)=0.$
We get that $M = max(f)=0$ and $m = min(f)=0$ and the Maximum Principle
implies that $f\equiv 0$ or $g=h.$ 
\end{proof}

\begin{remark}\label{rem:dp}
The example of Gambler's Ruin above is part of a larger set of
problems. Given a subset $B$ of nodes that we will call ``boundary''
and a function $h_{B}$ defined only on $B$, it is always possible to
extend this function on the remaining nodes (the {\em interior
points}) so that the extension $h$ is harmonic on the interior points.
Claim \ref{cl:uniq}, shows that if a solution exists it is
unique. However, the existence is obtained by writing down a solution
explicitly. For this we must introduce the notion of {\em stopping
times}. Given a walker starting at some interior node $x$ the stopping
time $\tau_{B}$ is the first time the walker visits a node in
$B$. Since the walk is a random process, stopping times are random variables.
The solution to the boundary-value (or Dirichlet) problem is
\[
h (x)\defeq \bE_{x} (h_{B} (X_{\tau_{B}}))
\]
namely, the expected value of $h_{B}$ evaluated at the exit point of a
walk started at $x$.
\end{remark}

\subsection{Matrices act on row vectors}

On a finite graph a random walker either runs until it hits a given
set of boundary points, as in the Gambler's ruin example, or it
bounces around forever. In the latter case, we can ask what 
fraction of time does it spend at a node $x$? 

Intuitively we want to define $\pi$ to be the {\em stable distribution} if
$\sum_{x\in V} \pi(x)=1$ and $\pi (x)$ denotes the probability of
finding the random walker at $x$ in the long run. This however already
implies that the long run stabilizes.

Let $\mu_0(x)$ denote the initial distribution for $X_0$ and
similarily $\mu_1(x)$ the distribution of $X_1$, after one 'step'. 
Then, conditioning on the previous location we get: 
\[
\mu_1(x)= \sum_y \mu_0 (y) \mathbb{P}(X_1=x|X_0=y) = \sum_y
\mu_0 (y) P(y,x).
\]
 
If we think of $\mu_0 (x)$ as a row vector then we can rewrite this in
matrix multiplication form as
$\mu_1=\mu_0 P$, $\mu_2 = \mu_1 P = \mu_0 P^2$, \dots . The entries of $P^2$
will be of the form 
\[
(P^{2})(x,y)= \sum_z P(x,z)P(z,y) =
\mathbb{P}(X_2=y|X_0=x).
\]

In general, we have that $\mu_n = \mu_0 P^n$, where $P^{n}$ is the
$n$th-power of the matrix $P$.

If, as $n$ goes to infinity,
$\lim_{n \rightarrow \infty} \mu_n = \pi$, then $\pi$ is a unique fixed-point
for $P$. That's because for large $n$
\[
\pi    \approx\mu_{n+1}=\mu_{n}P \approx \pi P
\]
We say that $\pi$ is a  {\em stable distribution} for $P$ if $\pi P=\pi$.

For a Markov chain on a 
finite state space there are very mild conditions (irreducibility and
aperiodicity) that 
guarantee the existence
and uniqueness of a stable distribution $\pi$, as well as the
convergence $\lim_{n\rightarrow \infty}\mu_0 P^{n}=\pi$ independently
of $\mu_0$. In such cases, $\pi (x)>0$ at every $x\in S$ and $1/
\pi(x)$ equals the expected return time to $x$ (in formulas
$\mathbb{E}_x (\tau_{x}^{+})$). Namely, $\pi$ is inversely
proportional to the average amount of time it takes for the random walker to
find its way back to the starting node. The easier it is to come back,
the larger $\pi$ is, and the larger the proportion of time the random
walker spends at $x$ in the long run.

\subsection{Reversible chains}\label{ssec:revchain}

The chains we will use in the sequel will have one extra property: 
reversibility.
\definition{The Markov chain $P$ is {\em reversible} if there is a distribution
$\pi$ such that 
\begin{equation}\label{eq:reverse}
\pi(x)P(x,y) = \pi(y)P(y,x)\qquad 
\forall x, y \in V.
\end{equation}
} 

If $P$ is reversible, then the distribution $\pi$ is stable, as
the following shows:
\[
\sum_{x}\pi (x)P(x,y)=\sum_{x}\pi (y)P(y,x)=\pi(y)\sum_{x}P(y,x)=\pi (y).
\]

The reversibility condition is very closely related to the notion of
``symmetric'' or ``self-adjoint'' matrices. 
In fact, if we let $A (x,y)=\pi (x) P (x,y)$,
then (\ref{eq:reverse}) is $A (x,y)=A (y,x)$.
In matrix notation $A=DP$, where $D$ is a diagonal matrix with entries
$\pi (x)$. So
standard results in linear algebra tell us that $A$ and thus $P$, are
linearizable to a diagonal matrix of real eigenvalues. Moreover, the
action of $P$ on row vectors is the action of the adjoint (or
transpose) $P^{T}$ on column vectors and $P^{T}$ shares the
same eingenvalues as $P$ and $A$. The difference lies in the
eigenvectors (or eigenfunctions), but there are simple formulas
involving $\pi$ to go from one set to the other.

The fact that the rows of $P$ sum to $1$ implies that the
eigenvalues of $P$ are between $-1$ and $1$. To see this note that
\begin{eqnarray*}
|Pf (x)| 
& = & |\sum_{y} P(x,y)f (y)| \\
& \leq  & (\max_{y} |f (y)|)\sum_{y}P (x,y)=\max_{y} |f (y)|
\end{eqnarray*}
Suppose that
$Pf=\lambda f$. Then, $|\lambda| |f (x)|\leq  \max_{y}|f (y)|$ for all
$x$'s. Hence, $|\lambda|\leq 1$.

Finally, the fact that the rows of $P$ sum to one also imply
straightforwardly that $1$ is an eigenvalue of $P$ for constant eigenvectors.

\begin{example}[Random walks on weighted graphs]
Let $ G = (V,E;W) $ be a weighted graph.
The degree of a vertex is usually the number of neighbors
of $x$, but for a weighted graph it can also be defined as
\[
\deg (x)\defeq \sum_{y\backsim x}W (x,y)
\]
Then a (weighted) random walk on $G$ is defined to be the Markov chain
with state space $V$ and transition matrix 
   $$P(x,y) = \left\{ \begin{array}{rl}\frac{W (x,y)}{\deg(x)} &\mbox{ if $y
\backsim x,$} \\ 0 &\mbox{ otherwise.}\end{array} \right.$$ 
 In this situation the probability distribution defined at each $x\in
V$ by:
   $$ \pi(x) \defeq  \frac{\deg(x)}{Z}\qquad \mbox{where $Z=\displaystyle\sum\limits_{z\in V}\deg(z)$}$$
will satisfy
\begin{eqnarray*}
\pi(x)P(x,y) & = &  \frac{\deg(x)}{Z}
\frac{W (x,y)}{\deg(x)}\\
& = & \frac{\deg(y)}{Z}\frac{W (y,x)}{\deg(y)}\\
& = &  \pi(y)P(y,x).
\end{eqnarray*}
In particular, $P$ is reversible in this case.
\end{example}

\section{Electrical Networks and Effective Resistance} \label{sec:electrical}

It is often useful when studying graphs and random walks on graphs to
think of them as electrical networks. Each edge $e$ then acquires a
weight $W (e)=C(e)$ which in this instance plays the role of {\em conductance}.
One recalls from high school physics that the electrical current $I$
through a connection is related to the potential difference $V$ via
the formula
\[
V=RI \qquad \mbox{and}\qquad I=C V
\]
where $R=1/C$ is the {\em resistance} and $C$ the conductance.

If instead of a simple connection one is looking at a series of
connections then simple rules allow one to say that the system behaves
as if there was only one connection but with an appropriately modified
resistance. This ``virtual'' resistance is what we call the {\sf
effective resistance} between two nodes.

\subsection{Flows and currents}

In order to make this more precise we need to introduce some notation.
Given two nodes $a$ and $b$ we consider a unit current that is
allowed to pass through a connected graph with the 'source' at
$a \in V$ and the 'sink' at $b \in V$. In order to effect this passage
a certain voltage difference must be applied at $a$ and $b$, and a
corresponding voltage potential will arise at every node $x$ in the graph.
Let $v(x)$ denote the voltage at $x \in V$.

The electric current that runs through the graph is an example of
a flow. Mathematically, {\em a flow} is an assignment of a number
$j_{xy}$ (representing its intensity) to every
directed edge
$(x,y)$, that has the property of being {\em  antisymmetric}, i.e.,
$j (x,y) =-j (y,x)$.
Given two nodes $a$ and $b$, we say $j$ is {\em a flow from $a$ to
$b$} if the following property is satisfied at every node $x\neq a,b$.

\noindent\underline{\bf Kirchoff's Node Law:} {\em The flow into a
vertex $x$ is equal to the flow out of $x$. In other words, the divergence
$\Div_{j}(x)\defeq\sum_{y \sim x} j(x,y)$ is equal to zero at every $x\neq
a,b$.}

Moereover, we also demand that $\Div_{j}(a) \geq 0$ (and thus $\Div_{j}(b)
\leq 0$), since it's a flow from $a$ to $b$. The quantity $\Div_{j} (a)$
is the {\em strength} of the flow $j$ from $a$ to $b$.

To compute effective
resistance we will assume that a unit flow is entering the network at $a$ and exiting at
$b$, so $\Div_{j} (a)=-\Div_{j} (b)=1$ and $\Div_{j} (x)=0$ for $x\neq a,b$.

A flow from $a$ to $b$ is furthermore a {\em current flow}, if it also
satisfies Ohm's 
law below. In this case we use the notation $i(x,y)$. 

\noindent\underline{\bf Ohm's Law:} {\em There is a potential function $v$ defined
on the nodes such that for every oriented edge $(x,y)$, $R (x,y) i (x,y) =
v(y)-v(x)$ where $R (x,y)$ is the resistance of the edge $(x,y)$.}

Whenever a
flow can be expressed 
as the edge difference of a potential function defined on the nodes, as
is the case for current, then around any cycle $x_{1}\sim\dots\sim
x_{n}\sim x_{n+1}=x_{1}$ the following holds necessarily:
\[
\sum_{k=1}^{n}R (x_{k},x_{k+1}) i(x_{k},x_{k+1})=0.
\]
This is known as \underline{\bf Kirchoff's Potential Law} and it turns
out to be equivalent to the existence of a potential.

Combining Kirchoff's Node Law and Ohm's Law at $x\neq a,b$, we get:
\begin{eqnarray*}
0 & = & \sum_{y \sim x} i(x,y)\\
& = & \sum_{y \sim x} \frac{v (y)-v (x)}{R (x,y)}\\
& = & \sum_{y \sim x} C (x,y) v(y)-C_{x}v (x),
\end{eqnarray*}
where $C (x,y)=1/R (x,y)$ is the conductance of the edge $(x,y)$ and
$C_{x}=\sum_{y\sim x}C (x,y)$ is the ``degree'' of node $x$ when using
conductance to weight each edge. In particular, we have that
\[
v (x)=\frac{1}{C_{x}}\sum_{y \sim x} C (x,y) v(y)
\]
namely the potential voltage function $v$ must necessarily be harmonic
with respect to the weighted graph.

That a current flow satisfying all of these laws and requirements
exists is a physical fact, however mathematically one has to prove its
existence and uniqueness. 

Proving uniqueness is fairly easy because flows that satisfy the two
Kirchoff laws mentioned above satisfy the
Superposition Principle, i.e. given two such flows $i_{1}$ and
$i_{2}$, then $j=i_{1}-i_{2}$ will also satisfy the same laws.

The existence proof is more challenging and its resolution
surprising. A flow can actually be constructed explicitly that
satisfies all the requirements using the concept of spanning trees.

Probabilistically, the resulting current $i(x,y)$ along an edge
$(x,y)$ can be shown to equal the expected number of (net) times that the random
walker (on the graph weighted by the conductances) crosses from $x$ to $y$.

\subsection{Effective resistance}

Given two nodes $a$ and $b$ consider a unit current flow entering at
$a$ and exiting the network at $b$. This flow exists and is unique, as
explained above, once the Neumann conditions, namely the entering and
exiting flow, is fixed at $a$ and
$b$. Moreover, this current determines a unique voltage potential $v$ at each
node (up to a constant). The absolute value of the voltage drop
between $a$ and $b$ is what we call the {\em effective resistance between
$a$ and $b$}. In formulas:
\begin{equation}\label{eq:effres}
\cReff (a,b)\defeq |v (a)-v (b)|=v (b)-v (a).
\end{equation}
\subsection{Effective resistance and escape probabilities}
Write 
\begin{equation}\label{eq:loccond}
C (x)\defeq\sum_{y\sim x}C(x,y),
\end{equation}
for the {\em local conductance} at a node $x$. Then the random walk on the weighted graph is the Markov chain with transition probabilities 
\begin{equation}\label{eq:condweights}
P(x,y)\defeq\frac{C(x,y)}{C(x)}
\end{equation}
whenever $y\sim x$.

Effective resistance, or more appropriately, {\em effective
conductance},
\begin{equation}\label{eq:effcond}
\cCeff (a,b)\defeq \frac{1}{\cReff (a,b)}
\end{equation}
is related to the the probability that a random walker starting at $a$
visits $b$ before returning to $a$. In symbols, recall the notion of
stopping time. Assuming that the random walk is at $a$ at time $0$, we write
$\tau_{z}$ for the first time it visits node $z$ and $\tau^{+}_{a}$
for the first time the walker revisits $a$ after time $1$. 
We are interested in $\bP_{a} (\tau_{z}<\tau^{+}_{a})$, the probability
that the random walk starting at $a$ visits $z$ before returning to $a$.

\begin{proposition}\label{prop:escprob}
For any $a,b\in V$,
\[
C (a)\bP_{a} (\tau_{b}<\tau^{+}_{a})=\cCeff (a,b).
\]
\end{proposition}
\begin{proof}
The proof is a beautiful application of the maximum principle for
harmonic functions (Fact \ref{fact:maxpr}). Consider $B=\{a,b\}$ to be the
boundary and $\Omega=V\setminus \{a,b \}$ the interior. We want to find a
harmonic function on $\Omega$ which takes the value $0$ at $a$ and the
value $1$ at $b$. By the maximum principle, the solution to this
problem is unique. We now produce two such solutions. 

The first one
is,
\[
h (x)\defeq \bP_{x} (\tau_{b}<\tau^{+}_{a}).
\]
Harmonicity can be checked by conditioning on the first step.

The second solution is given by normalizing the voltage function
$v(x)$ which is required in order to have one unit of current flow in
at $a$ and out at $b$. So let 
\[
g(x)\defeq \frac{v(x)-v(a)}{v(b)-v(a)}=\cCeff(a,b)(v(x)-v(a))
\]

By uniqueness, $h=g$. 

Now, by conditional probability,
\[
\begin{array}{lll}
\bP_{a} (\tau_{b}<\tau^{+}_{a}) & =  \sum_{x\in V}P(a,x))\bP_{x}
(\tau_{b}<\tau^{+}_{a}) & \mbox{(first-step analysis)}\\ 
&&\\
& =  \sum_{x\sim a} \frac{c(a,x)}{c(a)}\cCeff(a,b)(v(x)-v(a)) &
\mbox{(by  (\ref{eq:condweights}))}\\ 
&&\\
& =  \frac{\cCeff(a,b)}{C(a)}\sum_{x\sim a} i(a,x) & \mbox{(Ohm's Law)}\\
&&\\
& =  \frac{\cCeff(a,b)}{C(a)}\Div_{i}(a)=\frac{\cCeff(a,b)}{C(a)} &
\mbox{(by Kirchoff's node law)} 
\end{array}
\]
\end{proof}

\begin{remark}
The probability $p=\bP_{a} (\tau_{b}<\tau^{+}_{a})$ is also known as the {\em escape probability}.
With probability $1-p$ escape fails and the walker returns to $a$
before visiting $b$, at which point another identical walker starts
out for another attempt. This is akin to flipping a biased coin that
has probability $p$ of success (Heads). In particular, the number of
tosses $N$ required to achieve success is known to be distributed with
the geometric distribution: 
\[
\bP(N=k)=p(1-p)^{k-1}\qquad\mbox{for $k=1,2,3,\dots$}
\]
A calculation using the geometric series shows that the expected number of tosses $\bE(N)$ equals $1/p$.
In our context $N$ is the {\em number of visits to $a$ before
escaping through $b$} (where we count $t=0$ as a visit) and its
expectation is known as the {\em Green's function} of the random walk
started at $a$ and stopped at $b$. More generally, $G_b(a,c)$ is the
expected number of visits to $c$ for the walk started at $a$,
before it is stopped at $b$. It follows from this discussion and
Proposition \ref{prop:escprob} that $G_b(a,a)=C(a)\cReff(a,b)$. 
\end{remark}

\subsection{Rayleigh's monotonicity and energy}

The {\em energy} of a flow $j$ is
\[
\Energy (j)\defeq \sum_{e\in E} R(e) (j (e))^{2}
\]
Notice that for each edge $e\in E$, the quantity $j (e)$ is defined up
to a sign change. Therefore, the square $(j (e))^{2}$ is well-defined.

\begin{proposition}\label{prop:innerprod}
Consider the unit current flow $i$ from $a$ to $b$ and its
corresponding potential $v$ defined on the nodes. Given an arbitrary flow
$k$ from $a$ to $b$,
\[
\sum_{e\in E}R (e)i (e)k (e) =(v (b)-v (a))\Div_{k}(a).
\]
\end{proposition}
\begin{proof}
We have

\begin{align*}
\sum_{e\in E}R (e)i (e)k (e)    & = \frac{1}{2}\sum_{x\sim y, x,y\in V}R (x,y) i
(x,y)k (x,y) & \mbox{(double counting)}\\
& = \frac{1}{2}\sum_{x\sim y, x,y\in V} (v (y)-v (x))k(x,y)
&\mbox{(Ohm's law)}\\
& = \frac{1}{2}\sum_{y\in V} v (y)\sum_{x\in
V}k(x,y)-\frac{1}{2}\sum_{x\in V} v (x)\sum_{y\in V}k(x,y) & \mbox{($k
(x,y)=0$ if $x\not\sim y$)}\\
& =  -v (a)\Div_{k} (a)-v (b)\Div_{k} (b)& \mbox{(Kirchoff's node law)}\\
& = (v (b)-v (a))\Div_{k}(a)&\mbox{($\Div_k (b)=-\Div_{k} (a)$)} 
\end{align*}

\end{proof}
In particular, applying Proposition \ref{prop:innerprod} to the
case when $k$ equals 
the unit current flow $i$ itself and using
(\ref{eq:effres}), we find that
\begin{equation}\label{eq:energyeffres}
\Energy (i)=(v (b)-v (a))\Div_{i}(a)=\cReff (a,b).
\end{equation}
\begin{theorem}[Thomson's principle]\label{thm:thomson}
The unique unit flow from $a$ to $b$ that
minimizes energy is the unit current flow. 
\end{theorem}
\begin{proof}
Let $j$ be a unit flow from $a$ to $b$ and $i$ the unit current flow from $a$ to
$b$. Then $k\defeq j-i$ is also a flow from $a$ to $b$, but of
strength zero. We have

\begin{align*}
\Energy (j) & = \sum_{e\in E}R (e) (i (e)+k (e))^{2}\\
& = \Energy (i)+\Energy(k)+2\sum_{e\in E}R (e)i (e)k (e).
\end{align*}
Using Proposition \ref{prop:innerprod}, we see that the
cross-term $\sum_{e\in E}R (e)i (e)k (e)$ equals zero since $\Div_{k}(a)=0$. Therefore, the energy
of $j$ is strictly greater than the energy of $i$, unless $\Energy
(k)=0$ in which case $k\equiv 0$ and $i\equiv j$.
\end{proof}
\begin{corollary}\label{cor:rayleigh}
Effective resistance is a (not necessarily strictly) increasing
function of the edge resistances. 
\end{corollary}
\begin{proof}
Let $R (e)\leq R^{\prime} (e)$ for every $e\in E$, and let $i$ and
$i^{\prime}$ be the corresponding unit current flows from $a$ to $b$. Then

\begin{align*}
\cReff & =\sum_{e\in E}R (e) (i (e))^{2} & \mbox{By (\ref{eq:energyeffres})}\\
\ & \leq \sum_{e\in E}R (e) (i^{\prime} (e))^{2} & \mbox{(By Thomson's principle)}\\
\ & \leq \sum_{e\in E}R^{\prime } (e) (i^{\prime} (e))^{2} & \mbox{(By assumption)}\\
& = \cRpeff & \mbox{By (\ref{eq:energyeffres})}
\end{align*}

\end{proof}

\begin{example}
A {\em tree} $T$ is a connected graph with no loops. It follows that
given two vertices $x$ and $y$ on a tree, there is a unique {\sf
geodesic}, i.e.,  curve of minimal length. In this case the effective resistance and the graph distance coincide: 
\[
\cReff(x,y)=d_G(x,y),
\]
because no current can flow along edges that are not part of the unique geodesic.

Moreover, in a graph $G$ we always have 
\begin{equation}\label{eq:graphdistbound}
\cReff(x,y)\leq d_G(x,y).
\end{equation}
To see this, pick a shortest curve $\ga$ connecting $x$ to $y$ and note that it will not have any loops, since removing an edge along a loop does not affect the connectedness of $\ga$. Then complete $\ga$ to a {\em spanning tree} $T$ for $G$, namely a tree on the same $N$ vertices of $G$ which is also a subgraph of $G$. By monotonicity, Corollary \ref{cor:rayleigh}, we have $\cReff(x,y;G)\leq \cReff(x,y;T)$, because removing an edge from $G$ is equivalent to setting its resistance to be $\infty$. On the other hand, since $T$ is a tree $\cReff(x,y;T)=d_T(x,y)=d_G(x,y)$.
\end{example}
\begin{proposition}\label{effresmetr}
Effective resistance is a metric. 
\end{proposition}
\begin{proof}
Symmetry follows by reversing the current flow. Non-degeneracy follows
from (\ref{eq:energyeffres}) and the existence of a unit current flow. 
Finally, the triangle inequality is 
verified by noticing that inputing a unit current flow at $x$,
extracting it at $z$, then reinputing it at $z$ and extracting at $y$
is the same as just inputing it at $x$ and extracting it at $y$.  
\end{proof}

\subsection{Computing effective resistance with matrices}\label{ssec:effresmatrix}

\subsubsection{Adjacency matrix}\label{ssec:adjacency}

The {\em adjacency matrix} of a graph is a way to
represent which nodes of a graph are adjacent to each other. Given a
simple graph $G= (V,E)$, let  $A$ be the matrix with entries:
    $$A(i,j)= \left\{ \begin{array}{rl}1 &\mbox{  $(i,j)\in E,$} \\0
&\mbox{ otherwise.}\end{array} \right.$$
(We write $(i,j)\in E$ instead of $\{i,j\}\in E$ by abuse of notation, even though $(i,j)$ is an ordered pair.)

If the graph is weighted with weight $W(i,j)$, then we set
    \[A(i,j) = \left\{ \begin{array}{cl}W(i,j) &\mbox{ $(i,j)\in E,$}
\\0 &\mbox{ otherwise.}\end{array} \right.\] 

\subsubsection{Combinatorial Laplacian}

     Let $G$ denote a graph with vertex set
     $V = \{1,2,\dots,N\}$
     and edge set $E$. Then the {\em combinatorial Laplacian} $L$ is
an $N\times N$ matrix defined by
     $$L(i,j) = \left\{ \begin{array}{cl} d_{i} &\mbox{ if $i=j$,} 
\\-W(i,j) &\mbox{ if $(i,j)\in E$ }\\ 0 &\mbox{ otherwise }\end{array}
\right.$$
where $d_{i}$ is the degree of the vertex $i$, i.e.
\[
d_{i}\defeq \sum_{j}A(i,j).
\]
Letting  $D$ be the diagonal matrix 
$D=\Diag (d_{i})$, we see that $L = D- A$.

Also recalling the transition matrix $P$ for the random walk on $G$
defined in (\ref{eq:grphtrans}), we see that 
\[
L =D ( I - P ).
\]
\subsubsection{The square-root of the Laplacian}
Given an edge $e$ joining two nodes, arbitrarily assign one of these two nodes
to be $e$'s {\em tail} and the other one to be $e$'s {\em head}. Once
this choice is made we will write $e= (x,y)$ to mean that $x$ is the
tail and $y$ the head.

Consider the following  $M \times N$ matrix, where $M = |E|$,
and $N = |V|$,
\begin{equation}\label{eq:sqrtlapl}
B(e,x) \defeq  \left\{ \begin{array}{rl} 1 &\mbox{ if $x$ is $e$'s head }
\\-1 &\mbox{ if $x$ is $e$'s tail  }\\ 0 &\mbox{ otherwise }\end{array}
\right.
\end{equation}

We claim
that 
\begin{equation}\label{eq:lbtwb}
L = B^TWB
\end{equation}
where  $W$  is the $M\times  M$ diagonal matrix with entries $W (e)$
(the weight of edge $e$).  To check this, notice that
     \[ (B^{T}WB) (x,y) = \sum_{e} B^T(x,e)W (e) B(e,y)
 = \sum_{e}B(e,x)B(e,y)W(e).\]
If $x\neq y$, the only term that survives in the sum is when $x$ and $y$
are neighbors and $e= (x,y)$, then we get $-W (x,y)$. If $x=y$, any
edge from $x$ contributes $B (e,x)^{2}W (e)=W (e)$, so we get the
degree at $x$. 

\subsubsection{Quadratic Form}\label{sssec:quadform}
A {\em quadratic form} is a function of the form $Q(x)= x^TAx$, where
$A$ is a symmetric matrix. Since the combinatorial Laplacian is
symmetric we can write
\[
\begin{array}{ll}
 v^TLv & = \sum_{x,y \in V}v (x) L(x,y)v(y)\\
&\\
& = v^TB^TWBv = ||W^\frac{1}{2}Bv||^2_2 \\
&\\
& = \sum_e W(e)(Bv)(e)^{2} \\
&\\
& = \frac{1}{2}\sum_{(x,y) \in E}W (x,y) (v(x)-v(y))^2
\end{array}
\]
It follows that the quadratic form associated to the Laplacian is
positive semi-definite. In practice, this implies that the eigenvalues
are greater than or equal to zero, as can be seen from the so-called
``Minimax characterization'': let $\lambda_{1}\leq 
\lambda_{2}\leq \cdots \leq \lambda_{N}$, then
\[
\lambda_{k} = \min_{S: \Dim S=k} \max_{x \perp S}
\frac{x^TLx}{x^Tx}, 
\]
where $S$ is a linear subspace of $\R^N$.

\begin{remark}\label{rem:quadform}
If $v$ is the potential resulting from inputing a unit
current flow $i$ at $a$ and extracting it at $b$, then using Ohm's Law
and (\ref{eq:energyeffres}) we see that the quadratic form
satisfies:
\[
v^{T}Lv= \frac{1}{2}\sum_{(x,y) \in E}W (x,y) (v(x)-v(y))^2 =
\sum_{e}R (e)i (e)^{2}=\Energy (i)= \cReff (a,b).  
\]
\end{remark}

\subsubsection{The Kernel of the Laplacian}
We claim that the kernel of the Laplacian consists of the constant vectors:
     \[ \Ker(L) = \Ker(W^\frac{1}{2}B)= \Span\{[1\cdots 1]^T\}\] 
We use the quadratic form to check this:
\[
\begin{array}{ll}
Lv=0 & \Longleftrightarrow v^TLv = 0\\
&\\
& \Longleftrightarrow ||W^\frac{1}{2}Bv||_2^2 = 0 \\
&\\
& \Longleftrightarrow \sum_{x,y} W (x,y) (v(x)-v(y))^2 = 0 \\
&\\
& \Longleftrightarrow v(x)=v(y) \qquad \forall\  x,y.
\end{array}
\]
In particular, this means that the smallest eigenvalue of the
Laplacian is $0$. Traditionally, we relabel the eigenvalues from $0$ to
$N-1$, as follows: 
\begin{equation}\label{eq:eigenvalues}
\lambda_{0} = 0 < \lambda_{1} \leq \cdots \cdots
\leq \lambda_{N-1}.
\end{equation}

\subsubsection{Diagonalizing the Laplacian}

Being symmetric, the combinatorial Laplacian $L$ can be diagonalized,
i.e., there are eigenvectors $u_{0}$, $u_{1}$,\dots ,$u_{N-1}$ such that
\[
 L = \sum_{i=1}^{N-1}\lambda_{i}u_{i}u_{i}^T = U\Lambda U^T 
\]
Where $U=[u_{0}\ u_{1}\cdots u_{N-1}]$, $u_{0}= (1\cdots 1)^{T}$ and
 $\Lambda$ is the diagonal matrix of eigenvalues,

 \[\Lambda = 
\left[ \begin{array}{cccc}
0 & \cdots & \cdots & 0 \\
0 & \lambda_1 & \cdots & 0 \\
\vdots & \vdots & \ddots& \vdots \\
0 & 0 & \cdots & \lambda_{N-1} \end{array} \right]
\]

Recall that given a vector $u$, the matrix $uu^{T}$ is a rank one
matrix with range the line spanned by $u$.

Since the Laplacian has a  non-trivial kernel, it is not
invertible. However, we can still define a
pseudoinverse called the {\em Green operator}.  
     \[\cG =\sum_{i=1}^{N-1}\frac{1}{\lambda_{i}}u_{i}u_{i}^T.\]
  Then $\Ker (\cG)=\Ker (L)$ and
     \[ L\cG = \cG L = \sum_{i=1}^{N-1}u_{i}u_{i}^T,  \]
 which is also equal to the projection in $\R^{N}$  onto $\Span
(u_{1}, \dots, u_{N-1})$. 

\subsubsection{Kirchoff's and Ohm's Law revisited}

Let $I_{ext}$ denote the current injected at point $a$ and
extracted at point $b$, which can be thought of as an $N \times 1$ vector
     $$[ 0 \cdots 0\  -|I_{a}| \ 0 \cdots 0\ |I_{b}|\ 0
\cdots  0]^T.$$ 
 The current $i(e)$ for each edge $e$, can be written as an
$M\times 1$ vector since there are $M$ edges. Since each edge $e$ is
assigned a head and a tail, then $i(e)$ will either be positive or  negative  depending on whether the current flows from tail to head or vice versa. Then Kirchoff's node law can be written in matrix form  as
\begin{equation}\label{eq:currentlaw}
B^Ti = I_{ext},
\end{equation}
where $B$ is the square-root of the Laplacian defined in (\ref{eq:sqrtlapl}).

To see this check that at every vertex $x$,
     \[ I_{ext}(x) = \sum_{e\in E}B^T(x,e)i (e) =
\sum_{e\in E}B(e,x)i (e).\]
The only terms that survive in this sum correspond to edges $e$ with
either their head or tail at $x$. Let's first assume that each edge was oriented in such a way that the current
always flows from tail  to head. Then $B (e,x)=1$ for every edge with
current flowing in at $x$ and $B (e,x)=-1$ for every edge with current flowing out of $x$  (and $i(e)$ is always positive with this choice of
orientation). So the sum we get is exactly the divergence at $x$.  If for some reason an edge is oriented against the current flow then both $B(e,x)$ and $i(e)$ change sign so their product does not. 
    
Ohm's Law says that the resulting
voltage $v$ on the network satisfies 
     \[ i(x,y) =
\frac{v(y)-v(x)}{R(x,y)}=W (x,y) (v
(y)-v (x)). \] 
     In matrix form Ohm's Law can be written as: 
\begin{equation}\label{eq:ohmslaw}
i = WBv,
\end{equation}
as the following computation shows:
     \[(WBv) (e)=\sum_{z\in V}W(e)B(e,z)v(z) = W(e)(v(y)-v(x)),\]
where $e= (x,y)$, since $y$ is the head and $x$ the tail of $e$.

     Combining (\ref{eq:currentlaw}) and (\ref{eq:ohmslaw}) with
(\ref{eq:lbtwb}), we see that 
\[I_{ext} = B^TWBv = Lv.\]
We interpret this as an inhomogeneous problem to solve in $v$ with $I_{ext}$ given. However, as we've seen $L$ is not generally
invertible, unless $I_{ext}$ is perpendicular to the kernel of $L$,
i.e., $\Span\{ [1 \cdots  1 ]^T\}$. That's
exactly our case, since the current input at $a$ equals the current
output at $b$. Therefore, we can apply Green's operator to both sides
and get the
potential drop
\[
v = \cG  L v= \cG I_{ext}.
\]
In other words, writing 
\[
\eta_{ab}\defeq  [0\ \cdots 0\ -1 \ 0 \cdots 0\  1 \ 0 \cdots 0 ]^T = I_{ext},
\]
we get that the effective resistance
     \[\cReff(a,b) = v(b)-v(a) =\eta _{ab}^Tv = 
\eta_{ab}^T \cG \eta_{ab}.\] 
      If $a$ and $b$ are neighbors, then $\eta_{ab}^{T}$ is a row of $B$
corresponding to the edge $e= (a,b)$, and thus
     \[(B \cG B^T)(e,e) = \cReff(e).\]

\section{Concepts imported from function and potential theory}\label{sec:capmod}

\subsection{Capacity}

A function $u$ defined on the vertices induces a 
{\em gradient} $\rho_u$ on the edges, i.e., for $e=\{x,y\}$,
\[
\rho_u(e)\defeq |u(x)-u(y)|.
\]
The {\em energy} of $\rho_u$ is
\begin{equation}\label{eq:energy}
\cE(\rho_u)=\sum_{e\in E}\rho_u(e)^2.
\end{equation}
(On a weighted graph we modify this definition as follows: $\cE(\rho_u)=\sum_{e\in E}\rho_u(e)^2w(e)$.)

Given two nodes $a$ and $b$ in a graph $G$, we will minimize the energy among all the functions $u:V\rightarrow \R$ with $u(a)=0$ and $u(b)=1$. 
We define the {\em capacity} between $a$ and $b$ to be
\[
\capa(a,b)=\min_{u(a)=0,u(b)=1} \cE(\rho_u).
\]
The function $u$ that attains the minimum is called the {\em capacitary} function for $a$ and $b$.

Assuming that each edge has unit resistance, recall that Thomson's principle (Theorem \ref{thm:thomson}) and (\ref{eq:energyeffres}) imply that the effective resistance between $a$ and $b$ can be computed by minimizing the energy of all unit flows between $a$ and $b$, and that the minimum is achieved for the unit {\em current} flow. The electric potential $v$ that gives rise to the unit current flow can be used to interpret the unsigned flow as a "gradient". In fact, by Ohm's Law, for $e=\{x,y\}$,
\[
|i(e)|=|v(x)-v(y)|.
\]
The electric potential $v$ is only defined up to an additive constant, what matters is the drop $|v(a)-v(b)|$ which equals $\cReff(a,b)$. 

So, dividing $v$ by $\cReff(a,b)$ and shifting by a constant $C$ if necessary, we get a function
\begin{equation}\label{eq:capacitary}
U\defeq \frac{v}{\cReff(a,b)}+C
\end{equation}
such that $U(a)=0$, $U(b)=1$, and whose gradient is 
\[
\rho_U(e)=\frac{|i(e)|}{\cReff(a,b)}.
\]
Computing the energy of $\rho_U$ and using (\ref{eq:energyeffres}) we get an upper bound for $\capa(a,b)$:
\[
\capa(a,b)\leq \frac{1}{\cReff(a,b)}=\cCeff(a,b).
\]
This inequality is in fact an equality, even though the two sides are obtained by minimizing the energy of very different objects: flows that don't necessarily admit potentials, on one hand, and gradients which are obtained from a "potential" function defined on the vertex set, on the other hand.
\begin{proposition}\label{prop:capaceff}
We always have
\[
\capa(a,b)=\cCeff(a,b),
\]
and the function $U$ defined in (\ref{eq:capacitary}) is the capacitary function for $a$ and $b$.
\end{proposition}
\begin{proof}
The quadratic form induced by the combinatorial Laplacian of Section \ref{sssec:quadform} can be used to compute the energy of gradients, since
\[
u^TLu=\sum_{\{x,y\}\in E}|u(x)-u(y)|^2=\cE(\rho_u).
\]
In other words, if we number the nodes $1,\cdots,N$, the capacity $\capa(i,j)$ for $i< j$ is computed by minimizing the quadratic form restricted to the affine subspace of codimension $2$ 
\[
A_{ij}=\{x=(x_1,\dots,x_N)\in \R^N: x_i=0,\ x_j=1\}.
\]
This can be handled with the method of Lagrange multipliers. Let $g_i(x)=x_i$ for $i=1,\dots,N$ and $f(x)=x^TLx$.
Then given $i<j$,
\begin{equation}\label{eq:lagrange}
\nabla f=\lambda_i\nabla g_i+\lambda_j\nabla g_j
\end{equation}
for some parameters $\lambda_i$ and $\lambda_j$.

Notice that $\nabla g_i=e_i$, the standard unit vector in the $i$-th direction. Also $\nabla f(x)=2Lx$. Let $w=(w_1,\dots,w_n)$ be a solution of (\ref{eq:lagrange}). Then, interpreted as a function, $w$ is harmonic at every node except possibly for nodes $i$ and $j$. Moreover, $w_i=0$ and $w_j=1$. By the maximum principle, there is a unique harmonic function with these boundary values.  Therefore the solution to the Lagrange multipliers problem coincides with $U$ from (\ref{eq:capacitary}), i.e., the harmonic function obtained by renormalizing the electric potential $v$ arising from the effective resistance problem.
\end{proof}

\subsection{Modulus of curve families}

We relax the notion of gradient and consider arbitrary 
{\em $\rho$-densities}, i.e., functions $\rho: E\rightarrow [0,+\infty)$. The {\em $\rho$-length} of  a curve $\ga$ is then:
\begin{equation}\label{eq:rholength}
\rlength (\ga)=\sum_{e\in E(\ga)}\rho(e).
\end{equation}
We measure the energy of $\rho$ as done before in (\ref{eq:energy}) for gradients.

A {\em curve family}  $\Ga$ is a collection of curves $\ga$ in a graph $G$.
We say that $\rho$ is {\em admissible} for the curve family $\Ga$ if 
\begin{equation}\label{eq:admissible}
\rlength(\ga)\geq 1\qquad \forall \ga\in\Ga.
\end{equation}
We write $\cA$ for the family of all admissible densities for a given curve family $\Ga$.

The {\em modulus} of $\Ga$ is 
\begin{equation}\label{eq:modulus}
\Mod(\Ga)=\inf_{\rho\in \cA}\cE(\rho).
\end{equation}
The advantage of modulus is that any choice of admissible density gives rise to an upper-bound.
If the family $\Ga$ contains a constant curve, then its modulus is infinite. Otherwise, choosing $\rho\equiv 1$ we see that 
$\Mod(\Ga)$ is bounded above by $|E|$, the number of edges.

Given two nodes $a$ and $b$, let $\Mod(a,b)$ be the modulus of the curve-family consisting of all curves $\ga$ that contain both $a$ and $b$. We call this curve-family $\Ga(a,b)$.

It turns out that modulus and capacity are closely related concepts. 
\begin{proposition}\label{prop:modcap} We always have
\[
\Mod(a,b)=\capa(a,b).
\]
\end{proposition}
\begin{proof}
Let $U$ be the capacitary function for $a$ and $b$ defined in (\ref{eq:capacitary}), whose gradient $\rho_U$ has energy $\cE(\rho_U)$ equal to $\capa(a,b)$. We first show that $\rho_U$ is an admissible $\rho$-density for $\Ga(a,b)$. Let $\ga$ be an arbitrary curve from $a$ to $b$. Then, since $\ga$ is connected, it must contain a chain of vertices $x_0=a,\ x_1,\dots,x_m=b$ so that $x_j\sim x_{j+1}$ for $j=0,\dots,m-1$. By the triangle inequality:
\[
\begin{array}{ll}
1=|U(a)-U(b)| & = \left|\sum_{j=0}^{m-1}(U(x_{j+1})-U(x_j))\right|\\
&\\
& \leq \sum_{j=0}^{m-1}|U(x_{j+1})-U(x_j)|\\
&\\
& = \sum_{j=0}^{m-1}\rho_U(\{x_j,x_{j+1}\})\\
&\\
& \leq \sum_{e\in E(\ga)}\rho_U(e).
\end{array}
\]
So $\rho_U$ is admissible for $\Ga(a,b)$ and 
\[
\Mod(a,b)\leq \cE(\rho_U)=\capa(a,b).
\]

Conversely, let $\rho$ be an arbitrary admissible density for $\Ga(a,b)$. Without loss of generality we can assume that there is a curve $\ga_0\in \Ga(a,b)$ such that $\rlength(\ga_0)=1$, because otherwise we could scale $\rho$ by dividing by the shortest $\rho$-length and still have an admissible density. 

Define a function $u$ on the vertices by
\[
u(x)=\min_{\ga:a\leadsto x} \rlength(\ga),
\]
i.e., the shortest $\rho$-length of curves connecting $a$ to $x$.
Then $u(a)=0$ (the constant curve has length zero), and $u(b)= 1$ because of the curve $\ga_0$ mentioned above.

Furthermore, for an arbitrary edge $e=\{x,y\}$,
\[
u(y)\leq u(x)+\rho(e),
\]
because the shortest curve from $a$ to $x$ followed by the edge $e$ is a curve from $a$ to $y$.
Therefore, inverting the roles of $x$ and $y$, we find that the gradient of $u$ satisfies
\[
\rho_u(e)=|u(x)-u(y)|\leq \rho(e).
\]
This in turn implies that $\cE(\rho_u)\leq\cE(\rho)$, i.e.,
\[
\capa(a,b)\leq \cE(\rho).
\]
Since $\rho$ was an arbitrary admissible density, we can minimize over $\rho$ and get
\[
\capa(a,b)\leq \Mod(a,b).
\]

\end{proof}

Putting Propositions \ref{prop:capaceff} and \ref{prop:modcap} together we obtain:

\begin{theorem}\label{thm:modcap}
The three concepts of effective conductance, capacity and modulus coincide:
\[
\cCeff(a,b)=\capa(a,b)=\Mod(a,b).
\]
\end{theorem}
We can now exploit the definition of modulus as an infimum and obtain a comparison between the epidemic quasimetric and effective conductance.

\begin{proposition}\label{prop:upperbound}
\[
d_G(a,b)^2\Mod(a,b)\leq \epdist(a,b)
\]
\end{proposition}

\begin{corollary}\label{cor:estimate}
We have
\begin{equation}\label{eq:estimate}
\epdist(a,b)\geq d_G(a,b)^2\cCeff(a,b)=\frac{d_G(a,b)^2}{\cReff(a,b)}.
\end{equation}
\end{corollary}

\begin{proof}[Proof of Proposition \ref{prop:upperbound}]
Recall some of the notations introduced to define the epidemic quasimetric. We considered $\Om(a,d_G(a,b))$, i.e., the subgraph of $G$ induced by the vertices that are in the ball $\cB(a,d_G(a,b))$. Define a $\rho$-density for $\Ga(a,b)$ by letting $\rho(e)=1/d_G(a,b)$, if $e$ is an edge in $\Om\defeq\Om(a,d_G(a,b))\cup\Om(b,d_G(a,b))$, and let $\rho(e)=0$ otherwise. 
We claim that this $\rho$-density is admissible for $\Ga(a,b)$. To see this pick an arbitrary curve $\ga$ that contains both $a$ and $b$. By definition of graph distance (\ref{eq:graphdist}), we have $\length_G(\ga)\geq d_G(a,b)$, and this takes care of curves that stay in $\Om$. Some curves might actually exit $\Om$, but if they do then they must exit the ball $\cB(a,d_G(a,b))$ and therefore they will again have $\rho$-length greater than one.

By (\ref{eq:modulus}) we get that
\[
\Mod(a,b)\leq \frac{|\Om|}{d_G(a,b)^2}
\]
\end{proof}

\section{Explicit and numerical computations}\label{sec:numerical}

\subsection{Some closed form expressions}

We begin by calculating some of these quantities exactly, for special cases or families of cases.

\begin{example}[Path graphs]
For each $N\in \N$, $P_N$ is the unique graph on $N$ vertices that can be labeled $x_1,\dots,x_N$ so that $x_j\sim x_{j+1}$ for $j=1,\dots,N-1$. 

The diameter of $P_N$ is $N-1$.

Being a tree, effective resistance on $P_N$ is equal to the graph distance, i.e., for $i<j$
\[
\cReff(x_i,x_j)=d_G(x_i,x_j)=j-i.
\]

On the other hand,
\[
\epdist(x_i,x_j)=\min\{j-i,i-1\} + 2(j-i)+ \min\{j-i,N-j\}
\]
 To study how much is lost in  (\ref{eq:estimate})  we consider the "discrepancy":
\[
\delta(a,b)\defeq \frac{\cReff(a,b)\epdist(a,b)}{d_G(a,b)^2}.
\]

In the case of the path graph $P_N$, we have
\[
2\leq\delta(x_i,x_j)=\min\left\{1,\frac{i-1}{j-i}\right\}+2+\min\left\{1,\frac{N-j}{j-i}\right\}\leq 4.
\]
\end{example}

\begin{example}[The star and complete graphs]
The ratio $\delta$ can grow when the diameter of a tree is very small.

For instance, the star $S_N$ is a tree with a vertex $x_0$ that plays the role of hub and has $N-1$ neighbors. The diameter is $2$ and the graph distances are either $1$ or $2$. On the other hand, the epidemic quasimetric is either $N$ or $2(N-1)$. Therefore 
$\delta$ grows linearly with $N$.

The complete graph $K_N$ is the graph on $N$ vertices with the maximum number of edges, i.e., every vertex has $N-1$ neighbors. Here the diameter is $1$, and the graph distance is constant equal to $1$. Effective resistance can be computed as follow: fix two nodes $x$ and $y$ and let $\tilde{V}$ be the remaining nodes in $V\setminus\{x,y\}$. By symmetry the potential induced by a unit current flow between $a$ and $b$ will be constant on $\tilde{V}$. Hence no current will flow along the $K_{N-2}$ complete graph induced by $\tilde{V}$. We get that $\tilde{V}$ can be thought as a single vertex connected by $N-2$ edges to $x$ and $y$ respectively. Using the "serial" and "parallel" rules from high school physics, we find that
\[
\cReff(x,y)=\frac{2}{N}.
\]
The epidemic quasimetric in this case is constant and equal to $N(N-1)$. Therefore $\delta$ is also constant and equal to $(N-1)/2$. Again we see that $\delta$ grows linearly with $N$. 
\end{example}

One might conjecture that
the discrepancy $\delta$ always grows at most linearly in the number of vertices, but that turns out to be false. Using (\ref{eq:graphdistbound}), the volume of the complete graph $K_N$ and bounding $d_G(a,b)\geq 1$, we find that $\delta(a,b)\leq N^2$. This quadratic worst behavior can be achieved by letting $G$ be the graph obtained from the complete graph $K_N$ by picking a vertex $a$ and connecting it to a new vertex $b$ by a single edge. In this situation $d_G(a,b)=1$, and $\cReff(a,b)=1$, but $\epdist(a,b)=N(N-1)/2+2$, so $\delta(a,b)=O(N^2)$. One could say that in this case the epidemic quasimetric captures a feature of the graph $G$ that effective resistance does not see.

\subsection{Computing the epidemic quasimetric}

We wrote a simple routine in Matlab (actually we wrote a couple different ones). We are given a simple (possibly weighted graph) in the form of an adjacency matrix, see Section \ref{ssec:adjacency}. In some cases, matrices are too large and have to be entered in list form. This consists of two columns  of labels from $1$ to $N$, where each row represents an edge in the graph. 

We call the adjacency matrix $K$ and  normalize it so that all the non-zero entries are equal to $1$. The matrix $K$ is zero along the diagonal so we let $B=K+I$. We then take powers $B^k$ of $B$. An entry in the $(i,j)$ spot of $B^k$ is non-zero if and only if at least one of the terms in the sum
\[
\sum_{l_1,\dots,l_k} B(i,l_1)B(l_1,l_2)\cdots B(l_k,j)
\]
is non zero, i.e., if and only if there is at least one chain $(i,l_1), \dots, (l_k,j)$ connecting $i$ and $j$. Each step $(l_s,l_{s+1})$ either moves to a neighbor or stays at the given node.  We again normalize the entries of $B^k$ to be $1$ when non-zero. Let $B_k$ be the normalized version of $B^k$ and define $B_0=I$.

Next, given a node $t$ we look at the row $t$ in the matrix $B_k$ and find a $1$ in the column corresponding to each node that is graph distance less or equal to $k$ from $t$. We use this row to form a diagonal matrix $M$ and compute the matrix $K_{t,k}\defeq MKM$. The entries of $K_{t,k}$ are of the form 

\begin{equation}
K_{t,k}(i,j)=M(i,i)K(i,j)M(j,j) =  \left\{ \begin{array}{rl} 1 &\mbox{if $i,j\in \cB(t,k)$ and $\{i,j\}\in E$}
\\ 0 &\mbox{ else }\end{array}
\right.
\end{equation}
In other words, we see that  $K_{t,k}$ is the adjacency matrix of the subgraph of $G$ induced by the vertices that are in the ball $\cB(t,k)$ centered at $t$ of radius $k$. We now compute the volume by summing all the entries of $K_{t,k}$ and dividing by $2$. This is equivalent to first summing all the rows and getting the local degrees and then summing the rows and getting the sum of all the local degrees, which by the Handshake Lemma equals twice the number of edges.

Finally, recall that the graph distance between $x$ and $y$ is the first time $y$ belongs to the ball $\cB(x,k)$:
\[
d_G(x,y)=\min\{k: B_k(x,y)=1\}.
\] 
The way we actually compute the graph distance is as follows. We flip each zero and one in $B_k$, in practice we take a constant matrix $O$ with entries $O(x,y)=1$ and let $C_k(x,y)=O(x,y)-B_k(x,y)$. Then we sum the $C_k$ from $0$ to $N$, i.e., we let $D\defeq\sum_{k=0}^N  C_k$. Then 
\[
D(x,y)=\sum_{k=0}^{d_G(x,y)-1} 1=d_G(x,y).
\]

\subsection{Numerical experiments using the epidemic quasimetric in clustering algorithms}

The epidemic quasimetric carries more information for pairs of nodes that are close relative to the diameter of the graph.

In order to write the Matlab code and check it against a concrete graph it was really useful to have a very simple object such as the one in Figure~\ref{fig:simplegr}.

\begin{figure}[h]
\includegraphics[width=\textwidth]{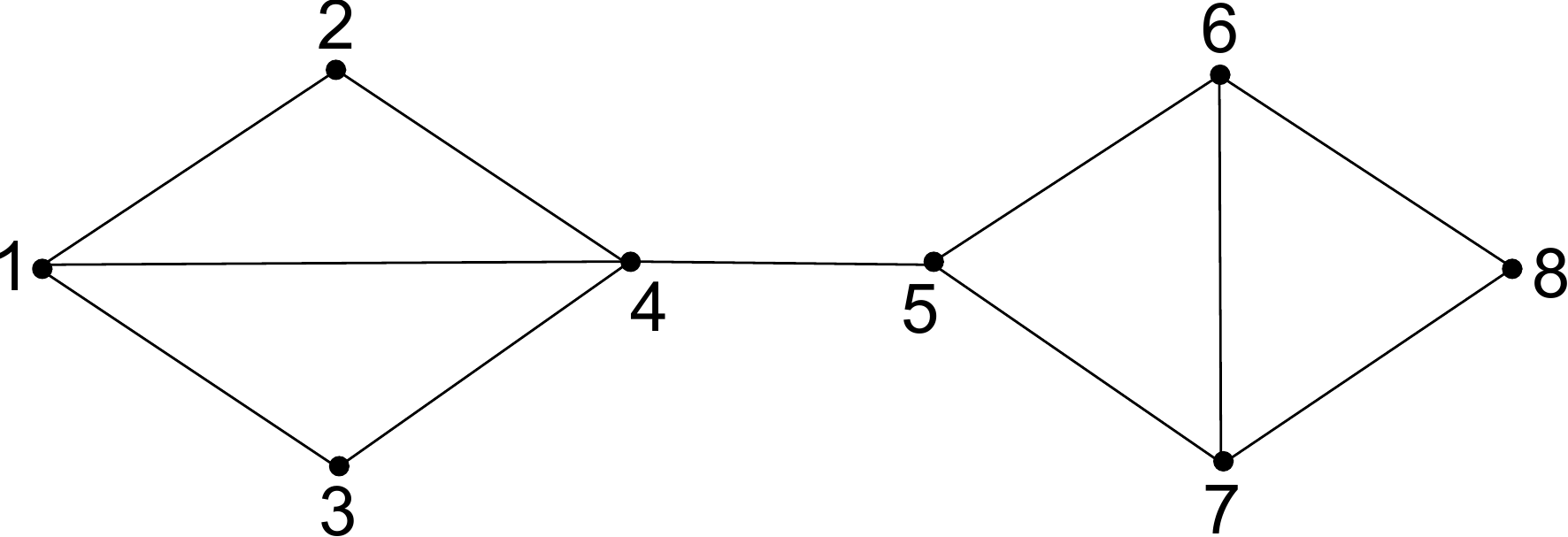}
\caption{A simple graph}\label{fig:simplegr}
\end{figure}

Once we wrote the algorithm for computing the epidemic quasimetric we tested how it would fare in the case of the famous ``Karate Club'', see Figure~\ref{fig:karateclub}. The nodes represent club members and the edges the friendship relations between them. After an argument between the two leaders, node 1 and 34, the club split into two clubs, thus the coloring of the vertices into red and blue vertices. The splitting is a real life phenomenon and then the intuition is that the web of friendships should play a role.

\begin{figure}[h]
\includegraphics[width=\textwidth]{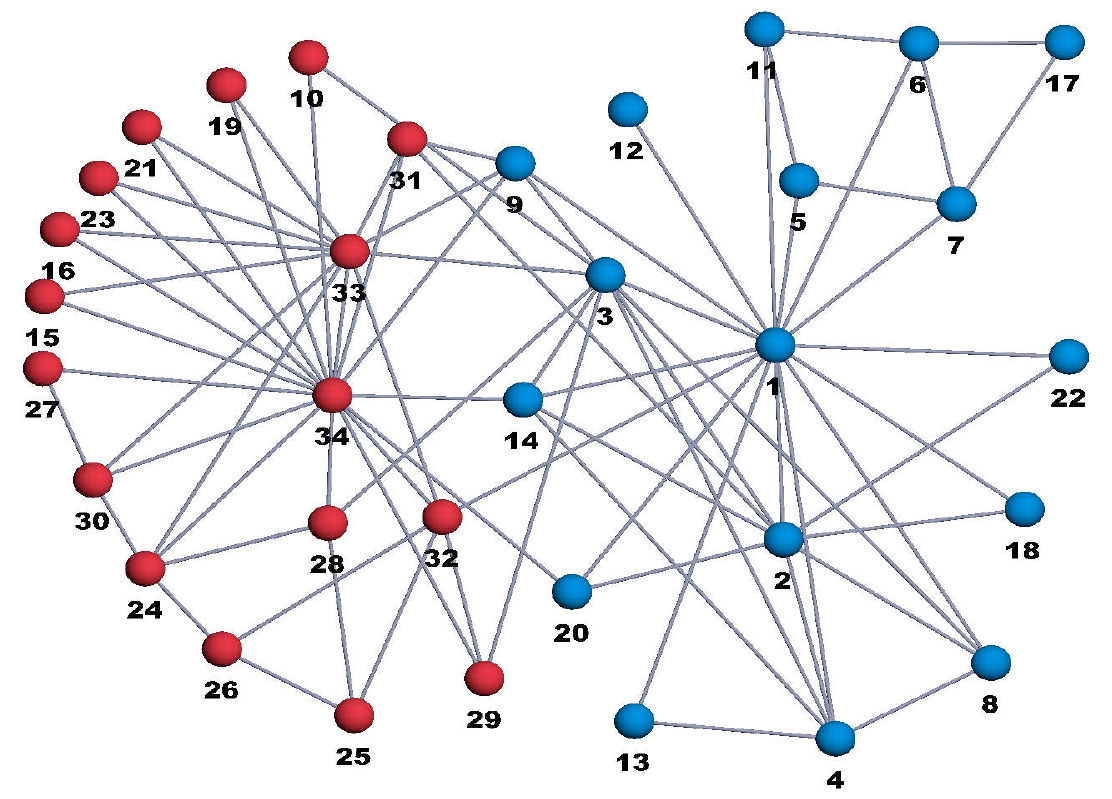}
\caption{The original split in the Karate Club}\label{fig:karateclub}
\end{figure}

We used the epidemic quasimetric as a measure of ``similarity'' between nodes in the agglomerative AGNES algorithm (this we performed in R). We obtained the dendrogram in Figure~\ref{fig:dendrogram}. This is an algorithm that begins by putting each vertex in its own class. It then `agglomerates' two classes based on the degree of similarity between them, which is computed by taking averages of all the similarities between elements of the two classes. The algorithm tries to agglomerate the least dissimilar classes first. 

\begin{figure}[h]
\includegraphics[width=\textwidth]{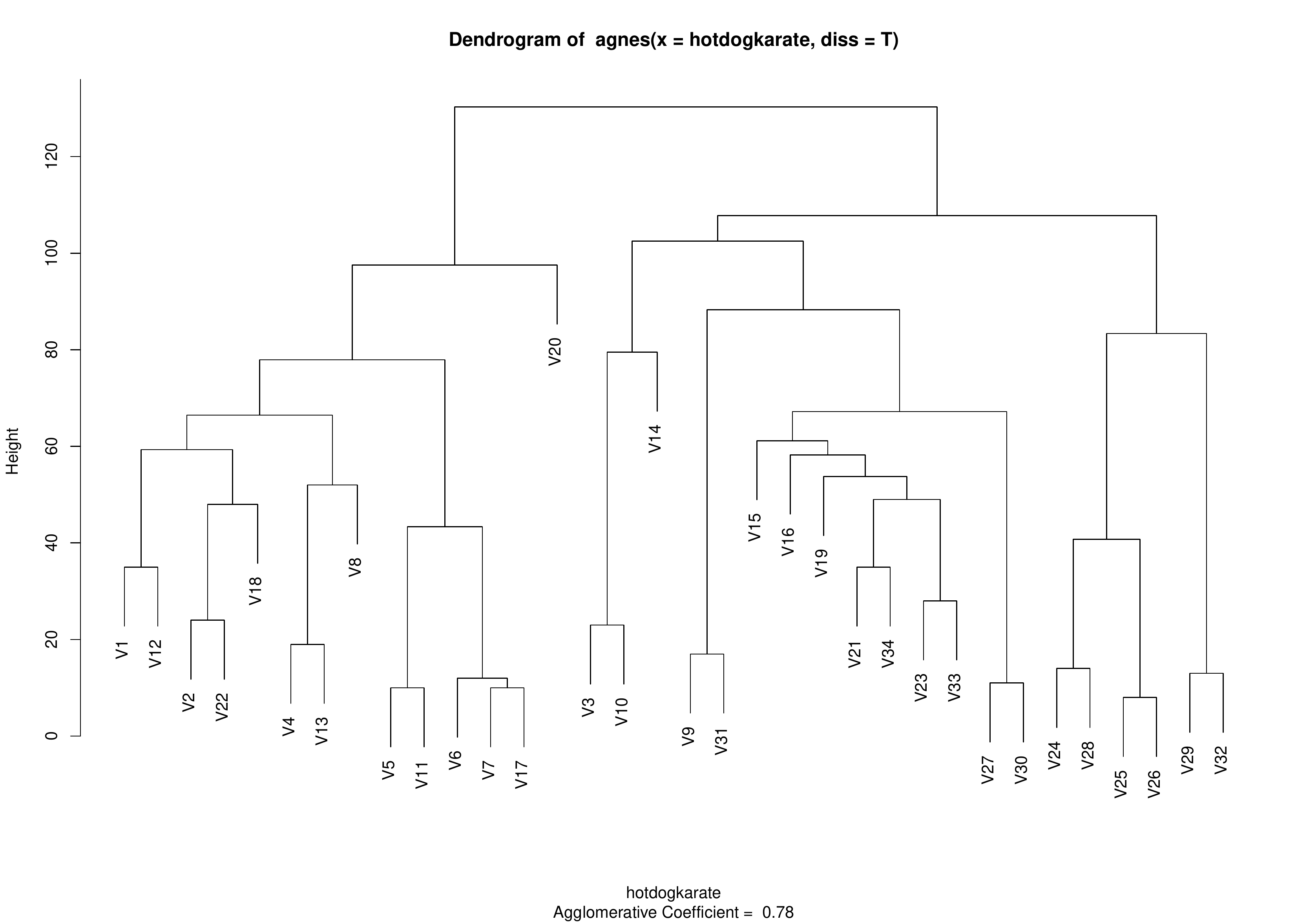}
\caption{AGNES produced this dendrogram}\label{fig:dendrogram}
\end{figure}
Using the two largest clusters we obtained a cut for the Karate club that mislabels only three vertices as can be seen in Figure \ref{fig:karateepidemiccut}.

\begin{figure}
\includegraphics[width=\textwidth]{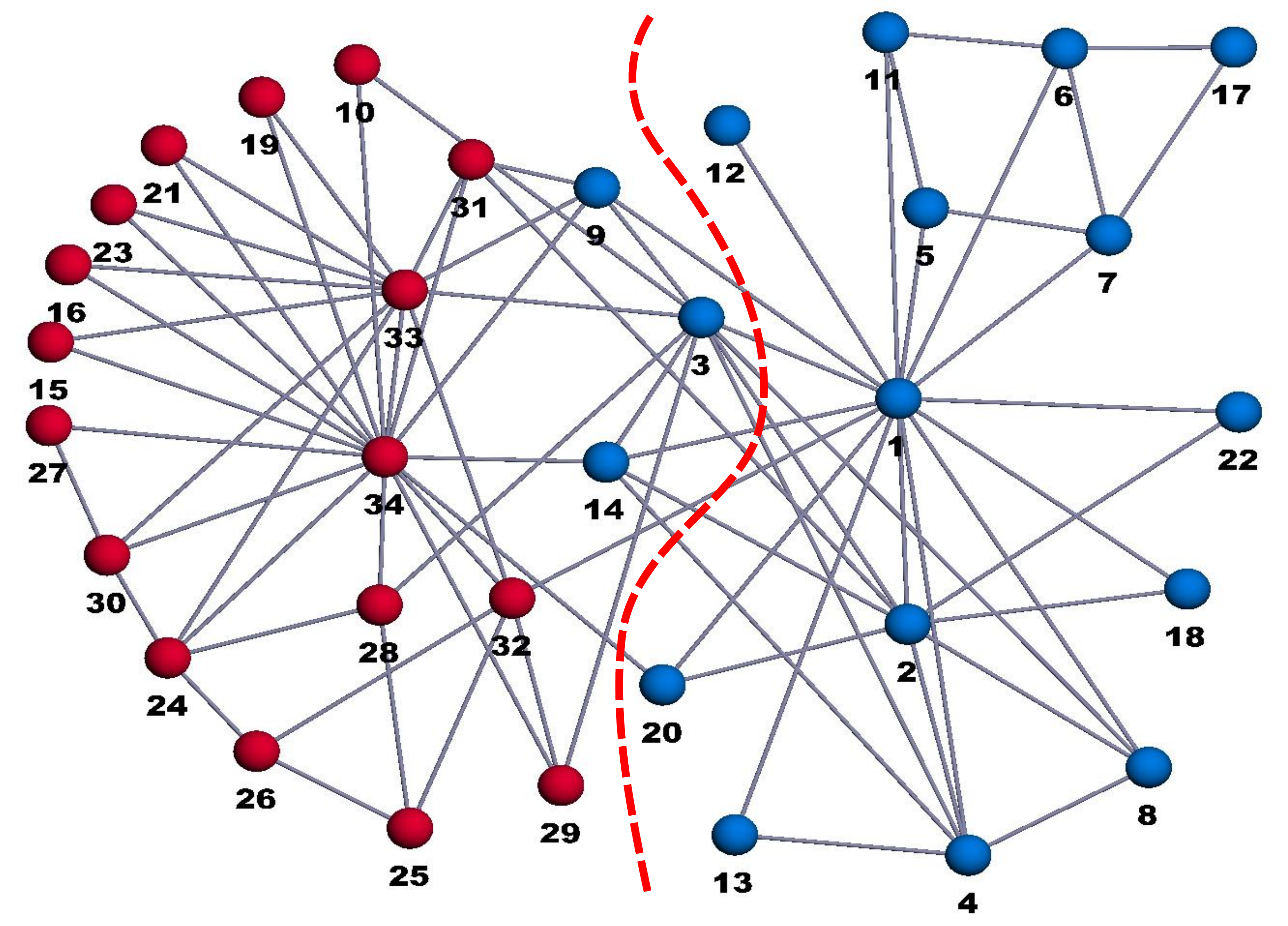}
\caption{Epidemic cut of the Karate Club}\label{fig:karateepidemiccut}
\end{figure}

\newpage

\def\cprime{$'$}

\end{document}